\documentclass[11pt]{amsart}
\usepackage{amssymb,amsmath,amsfonts,amscd,euscript}

\newcommand{\nc}{\newcommand}

\numberwithin{equation}{section}
\newtheorem{thm}{Theorem}[section]
\newtheorem{prop}[thm]{Proposition}
\newtheorem{lem}[thm]{Lemma}
\newtheorem{conj}[thm]{Conjecture}
\newtheorem{cor}[thm]{Corollary}
\theoremstyle{remark}
\newtheorem{rem}[thm]{Remark}
\newtheorem{definition}[thm]{Definition}
\newtheorem{example}[thm]{Example}
\newtheorem{dfn}[thm]{Definition}

\nc{\gl}{\mathfrak{gl}}
\nc{\GL}{\mathfrak{GL}}
\nc{\g}{\mathfrak{g}}
\nc{\gh}{\widehat\g}
\nc{\h}{\mathfrak{h}}
\nc{\la}{\lambda}
\nc{\al}{\alpha }
\nc{\be}{\beta }
\nc{\ve}{\varepsilon }
\nc{\om}{\omega }

\nc{\ta}{\theta}
\nc{\veps}{\varepsilon}
\nc{\ch}{{\mathop {\rm ch}}}
\nc{\Gr}{{\mathop {\rm Gr}}}
\nc{\Tr}{{\mathop {\rm Tr}\,}}
\nc{\Id}{{\mathop {\rm Id}}}
\nc{\ad}{{\mathop {\rm ad}}}
\nc{\bra}{\langle}
\nc{\ket}{\rangle}
\nc{\x}{{\bf x}}
\nc{\bs}{{\bf s}}
\nc{\bp}{{\bf p}}
\nc{\bc}{{\bf c}}
\nc{\pa}{\partial}
\nc{\ld}{\ldots}
\nc{\cd}{\cdots}
\nc{\hk}{\hookrightarrow}
\nc{\T}{\otimes}
\newcommand{\bea}{\begin{equation}}
\newcommand{\ena}{\end{equation}}
\nc{\gr}{\mathrm{gr}}
\nc{\ov}{\overline}

\nc{\cO}{\mathcal O}
\nc{\cF}{\mathcal F}
\nc{\cL}{\mathcal L}
\nc{\msl}{\mathfrak{sl}}
\nc{\msp}{\mathfrak{sp}}
\nc{\mgl}{\mathfrak{gl}}
\nc{\U}{\mathrm U}
\nc{\V}{\EuScript V}
\nc{\bH}{\EuScript H}
\nc{\Res}{\mathrm{Res\ }}

\newcommand{\bC}{{\mathbb C}}
\newcommand{\bZ}{{\mathbb Z}}

\newcommand{\bP}{{\mathbb P}}
\newcommand{\bG}{{\mathbb G}}
\newcommand{\BA}{{\mathbb A}}

\newcommand{\fa}{{\mathfrak a}}
\newcommand{\fg}{{\mathfrak g}}
\newcommand{\fh}{{\mathfrak h}}
\newcommand{\fb}{{\mathfrak b}}

\newcommand{\fn}{{\mathfrak n}}

\newcommand{\Fl}{\EuScript{F}}

\begin{document}

\title{Degenerate affine Grassmannians and loop quivers}
\author{Evgeny Feigin, Michael Finkelberg, Markus Reineke}
\address{Evgeny Feigin:\newline
Department of Mathematics,\newline National Research University Higher School of Economics,\newline
Russia, 117312, Moscow, Vavilova str. 7\newline
{\it and }\newline
Tamm Department of Theoretical Physics,
Lebedev Physics Institute
}
\email{evgfeig@gmail.com}
\address{Michael Finkelberg:\newline
National Research University Higher School of Economics,\newline
Russia, 117312, Moscow, Vavilova str. 7; \newline 
Institute for Information Transmission Problems of RAS, \newline
Russia, 127051, Moscow, Bolshoj Karetnyj per. 19}
\email{fnklberg@gmail.com}
\address{Markus Reineke:\newline
Fachbereich C - Mathematik, Bergische Universit\"at Wuppertal, D - 42097 Wuppertal}
\email{mreineke@uni-wuppertal.de}

\begin{abstract}
We study the connection between the affine degenerate Grassmannians in type $A$, quiver Grassmannians for one vertex loop quivers and affine
Schubert varieties. We give an explicit description of the degenerate affine Grassmannian of type $GL_n$ and identify it
with semi-infinite orbit closure of type $A_{2n-1}$. We show that principal quiver Grassmannians for the one vertex loop quiver
provide finite-dimensional approximations of the degenerate affine Grassmannian. Finally, we give an explicit
description of the degenerate affine Grassmannian of type $A_1^{(1)}$,  propose
a conjectural description in the symplectic case and discuss the generalization to the case of the affine degenerate flag varieties.

\end{abstract}

\maketitle

\section*{Introduction}
Let $G$ be a simple Lie group and let $G/B$ be the flag variety attached to $G$. These varieties enjoy many nice properties;
in particular, they are spherical, i.e.~the Borel subgroup acts on a flag variety with an open orbit. The varieties
$G/B$ can be degenerated in such a way that the action of the Borel subgroup degenerates (modulo torus) into an action of the abelian
unipotent group of the same dimension, acting with an open orbit on the degenerate flag variety (see\cite{F1}, \cite{F2}).
The construction is of Lie theoretic nature and uses the theory of highest weight $G$-modules.

It was shown in \cite{CFR1} that in type $A$ the degenerate flag varieties are closely related to the representation
theory of the equioriented quivers of type $A$. More precisely, the degeneration of $SL_n/B$ is isomorphic to a certain quiver
Grassmannian of subrepresentations of a direct sum of an injective and projective representations. This observation, on the one hand,
allows to use the representation theory of quivers in order to study the geometry of the degenerate flag varieties, and,
on the other hand, produces links between Lie theory and quiver Grassmannians.

Both the degenerate flag varieties and the quiver Grassmannians are known to be related to the theory of Schubert varieties. More precisely,
one can identify certain quiver Grassmannians and degenerate flag varieties with certain Schubert varieties (see \cite{CL}, \cite{G}, \cite{PRS}).
The main question we ask in this paper is to what extent one can generalize the degeneration constructions above to the case of the affine Lie
algebras and affine Lie groups? Is there a connection with the theory of quivers and with Schubert varieties in this case?
We give partial answers to the questions above. In short, the affine story is rather complicated (even in type A);
the correct replacement of the equioriented type $A$ quiver is the (equioriented) cycle quiver.

In our paper we concentrate on the case of the degenerate affine Grassmannians (though the general case is also discussed).
From the point of view of representation theory this means that we restrict to the basic level one module of the affine Lie algebras.
The general definition of the degenerate flag varieties goes through the theory of highest weight modules. In the
affine case the same definition works perfectly. Unfortunately, we are not able to identify the resulting ind-variety
even for affine $\msl_n$ (however, we completely describe the $A_1$ case and put forward a conjecture for the symplectic Lie
algebras). The problems pop up even on the level of representations: we do not have good enough description of the 
PBW graded level one basic representation $L_0^a$ of the affine $\msl_n$. In section 2 we formulate a conjecture saying
that $L_0^a$ can be realized inside the semi-infinite wedge space. Unfortunately, we can not prove the conjecture at the moment.  
  
It turns out that there exists a similar object attached to the affine Lie algebra $\widehat{\mgl}_n$.
More precisely, using the formalism of the semi-infinite wedge spaces, we define the PBW degeneration ${\rm Gr}^a(\mgl_n)$ of the affine
Grassmannians. The definition is very similar to the general Lie theoretic one and we are able
to describe the degenerate object in linear algebra terms. Using this description, we make a connection to the theory
of quiver Grassmannians for the one vertex loop quivers. More precisely, ${\rm Gr}^a(\mgl_n)$  can be realized as an inductive
limit of quiver Grassmannians, which are analogues of Schubert varieties. We use the representation theoretic techniques
to derive algebro-geometric properties of this finitization. In particular, we study the structure of orbits and construct desingularizations
explicitly. We also identify them with certain classical affine Schubert varieties for larger groups.

Yet another approach to the study of the $\widehat{\mgl}_n$ degenerate Grassmannians comes from the identification with
the closure of the semi-infinite orbit for the classical parabolic flag varieties for $\widehat{\msl}_{2n}$.
We identify the finitization as above with concrete subvarieties inside the closure.
We thus realize ${\rm Gr}^a(\mgl_n)$ as a semi-infinite orbit closure inside the $\msl_{2n}$ classical affine Grassmannian.

The paper is organized as follows. In section 1 we collect the basic objects and constructions on affine algebras and
PBW degenerations. In section 2 we define and study degenerate affine Grassmannians of type $A$.  Section 3 is devoted to
the quiver part of the story: the principal quiver Grassmanians for the one vertex loop quivers are studied and the identification
with subvarieties in the degenerate  flag varieties and with affine Schubert varieties is constructed.
In section 4 we discuss the degenerate affine Grassmannians  of types $\widehat{\msl}_2$ and $\widehat{\msp}_{2n}$.

\section{The setup}
\subsection{Affine Lie algebras}
Let $\g$ be a simple Lie algebra, $\gh$ the corresponding affine algebra. We have
\[
\gh=\g\T\bC[t,t^{-1}]\oplus\bC K\oplus \bC d,
\]
where $K$ is central and $[d,x\T t^i]=-ix\T t^i$.
Let $\g=\fn^-\oplus\fh\oplus\fn$ be the Cartan decomposition. Consider the decomposition
for the affine algebra $\widehat{\fg}=\widehat{\fg}^-\oplus\widehat{\fg}^0\oplus\widehat{\fg}^+$, where
\begin{gather*}
\widehat{\fg}^-=\fn^-\T 1\oplus \fg\T t^{-1}\bC[t^{-1}], \ \widehat{\fg}^0=\fh\T 1\oplus\bC K\oplus \bC d,\\
\widehat{\fg}^+=\fn^+\T 1\oplus \fg\T t\bC[t].
\end{gather*}

Let $\theta$ be the highest root for $\fg$.
For a dominant integral $\g$-weight
$\la$ and a non-negative integer $k$ such that $(\la,\theta)\le k$, let $L_{\la,k}$ be the
corresponding irreducible integrable highest weight $\gh$-module with a highest weight vector
$v_{\la,k}$. We have
\begin{gather*}
\widehat{\fg}^+ v_{\la,k}=0,\ \U(\widehat{\fg}^-)v_{\la,k}=L_{\la,k},\\
(h\T 1) v_{\la,k}=\la(h)v_{\la,k},\ Kv_{\la,k}=kv_{\la,k},\ dv_{\la,k}=0.
\end{gather*}

Let $G$ be a simple simply-connected Lie group with a Borel subgroup $B$, $\fg=Lie(G)$, $\fb=Lie(B)$.
Let $\widehat G$ and $I$ be the corresponding affine group and its Iwahori subgroup, respectively.
For a parahoric subgroup $P\subset \widehat G$ such that $I\subset P\subset \widehat G$,
let $\widehat G/P$ be the corresponding affine flag variety. These varieties are infinite-dimensional
ind-varieties, i.e.~they are inductive limits of the finite-dimensional Schubert varieties. More precisely,
let $T\subset\widehat G$ be the Cartan torus and let $p$ be a $T$-stable point in the affine flag variety. The corresponding Schubert variety
is the closure of the $I$-orbit through $p$. If $P=P_0$ is  the maximal parahoric corresponding to the affine simple root
then the corresponding flag variety is called the affine Grassmannian. Finally, we note that if $P$ stabilizes the highest weight line
$\bC v_{\la,k}$ in $\bP(L_{\la,k})$, then we have a natural embedding $\widehat G/P\subset \bP(L_{\la,k})$.

\subsection{Sato Grassmannians and flag varieties of type $A_\infty$}
We consider an infinite-dimensional vector space $V$ with a basis $v_i$, $i\in\bZ$. The semi-infinite
wedge space $F=\Lambda^{\infty/2}$ is spanned by the elements
\[
v_{i_0}\wedge v_{i_1}\wedge\dots,\ i_0>i_1>\dots,\ i_{k+1}=i_k-1 \text{ for } k \text{ large enough.}
\]
The charge of the wedge product $v_{i_0}\wedge v_{i_1}\wedge\dots$ is defined as $(i_k+k)$ for $k$ large enough.
We have the decomposition $F=\bigoplus_{m\in\bZ} F^{(m)}$, where $F^{(m)}$ is spanned by the wedge products of charge $m$.
We define the vacuum vectors
\[
|m\ket=v_m\wedge v_{m-1}\wedge v_{m-2}\wedge\dots\in F^{(m)}.
\]

The Lie algebra $\mgl_\infty$ is spanned by the matrix units $E_{i,j}$, i.e. it consists of the infinite matrices with finite support.
The natural action of $\mgl_\infty$ on the space $V$ extends to an action on each space  $F^{(m)}$. It is easy to see that
each $F^{(m)}$ is an irreducible $\mgl_\infty$ module with highest weight vector $|m\ket$.

The Sato Grassmannian ${\rm SGr}_m$ sits inside $\bP(F^{(m)})$ via the Pl\"ucker embedding; it consists of subspaces $U\subset V$ with the property that there
exists $N\in\bZ$ such that
\[
{\rm span}(v_N,v_{N-1},\dots)\subset U \text{ and } \dim U/{\rm span}(v_N,v_{N-1},\dots)=m-N.
\]
In particular, the line $\bC|m\ket$ belongs to (the image of) ${\rm SGr}_m$. It is easy to see that all the varieties
${\rm SGr}_m$ are isomorphic (via shifts of the indices of  $v_k$).
Following \cite{KaPe} we define the flag variety  $\Fl_\infty$ of type $A_\infty$ as the subvariety of the product
$\prod_{m\in\bZ} {\rm SGr}_m$ consisting of
collections $(U_m)_{m\in\bZ}$, $U_m\in {\rm SGr}_m$ such that $U_m\subset U_{m+1}$.

\subsection{Type $A^{(1)}_{n-1}$ case}
Now let us consider an $n$-dimensional vector space $W$ with a basis $w_1,\dots,w_n$.  Let us identify
the space $W\T \bC[t,t^{-1}]$ with $V$ by
\begin {equation}\label{fold}
v_{nk+j}=w_j\T t^{-k-1},\ j=1,\dots,n,\ k\in\bZ
\end{equation}
In particular,
\[
|0\ket=(w_1\T 1)\wedge\dots\wedge (w_n\T 1)\wedge (w_1\T t)\wedge\dots \wedge (w_n\T t)\wedge\dots .
\]
This gives an embedding
$\msl_n\T\bC[t,t^{-1}]\subset\mgl_\infty$ and induces an action of $\widehat{\msl}_n$ on each $F^{(m)}$.
The irreducible highest weight  $\widehat{\msl}_n$ modules can thus be realized in the semi-infinite picture.
In particular, the level one module $L_{\omega_i,1}$, $i=0,\dots,n-1$ can be seen as the subspace of $F^{(i)}$ generated by $|i\ket$.
We also see that the affine Grassmannian $\widehat{SL_n}/P_0$ of type $A_n$ is naturally embedded into the Sato Grassmannian.
The image of this embedding can be described explicitly:
\begin{equation}
\widehat{SL_n}/P_0 = \{U\in {\rm SGr}_0:\ tU\subset U\}.
\end{equation}
The complete flag variety for affine $SL_n$
is formed by the collections $(U_i)_{i\in\bZ}$, $U_i\in {\rm SGr}_i$ such that $U_i\subset U_{i+1}$, $tU_i\subset U_{i+n}$.

\subsection{PBW filtration and degenerate flag varieties}
We first collect some facts on the PBW degeneration in type $A$ (\cite{FFoL1}, \cite{FFoL3}, \cite{F1}). Given a highest weight representation $V_\la$
of $\msl_n$ we have an increasing filtration on $V_\la$, induced by the PBW filtration on $\U(\fn^-)$. The associated graded
spaces $V^a_\la$ are modules over the abelian algebra $(\fn^-)^a$ with the underlying vector space $\fn^-$.
The degenerate flag variety $\Fl^a_\la(\fg)\subset \bP(V^a_\la)$ is the closure of the orbit of the group
$\exp((\fn^-)^a)$ through the line containing the highest weight vector $v_\la$.

Let $\la$ be a regular dominant weight for $\fg=\msl_n$. Then the corresponding degenerate flag variety does not depend on $\la$
and can be described explicitly as follows.
Let $W=\bC^n$ with a basis $w_1,\dots,w_n$ and let $pr_k$ be the projection along $w_k$ to the span
of the remaining basis vectors. Then the degenerate flag variety attached to a regular dominant weight 
consists of collections
$(V_1,\dots,V_{n-1})$, $V_k\in{\rm Gr}_k(W)$ such that $pr_{k+1}V_k\subset V_{k+1}$. We denote this variety by $\Fl^a(\msl_n)$.
Note that it was shown recently in \cite{CL}, Theorem 1.2  that $\Fl^a(\msl_n)$ is isomorphic to a Schubert variety for $\msl_{2n-2}$. 
We also note that the
degenerate Grassmann varieties for $\msl_n$ coincide with their classical analogues.

The explicit realization of $\Fl^a(\msl_n)$ can be naturally generalized to the $A_\infty$ case.
Recall the basis $v_i$ of $V=\bC^\infty$. We define the operators $pr_k:V\to V$ as projections along $v_k$ to the span of
$v_i$, $i\ne k$.  The following definitions are degenerate versions of the classical analogues (see e.g. \cite{KaPe}).
\begin{dfn}
The full degenerate flag variety $\Fl^a_\infty$ of type $A_\infty$ is the subvariety of the product $\prod_{m\in\bZ} {\rm SGr}_m$ consisting of
collections $(U_m)_{m\in\bZ}$, $U_m\in {\rm SGr}_m$ such that $pr_{m+1}U_m\subset U_{m+1}$.
\end{dfn}
We note that these varieties can be seen as infinite limits of the type $A_n$ degenerate flag varieties.

\begin{rem}
The degeneration procedure from $\Fl_\infty$ to $\Fl^a_\infty$ can be described as follows. 
Let ${\bf F}_\infty$ be the subvariety inside ${\mathbb A}^1\times \prod_{m\in\bZ} {\rm SGr}_m$, consisting 
of the points $\left(t,(U_m)_{m\in\bZ}\right)$ such that $pr_{m+1}(t) U_m\subset U_{m+1}$, where
$pr_{m+1}(t):V\to V$ is the map defined by $pr_{m+1}(t)v_k=v_k$, $k\ne m+1$ and $pr_{m+1}(t)v_{m+1}=tv_{m+1}$.
Then we have the natural projection ${\bf F}\to{\mathbb A}^1$. The fiber over $t=0$ is isomorphic to $\Fl^A_\infty$ and
the general fiber is isomorphic to $\Fl_\infty$.
\end{rem}

Let us now define the affine degenerate flag varieties of type $\mgl_n$.
We denote by $pr_{w_i\T t^k}:W\T \bC[t,t^{-1}]\to W\T \bC[t,t^{-1}]$ the projections along $w_i\T t^k$
to the span of the remaining vectors of the form $w_j\T t^l$.

\begin{dfn}
The degenerate flag variety $\Fl^a(\mgl_n)$ of type $\mgl_n$  consists of collections
$(U_i)_{i=0}^n$ such that $U_i\in{\rm SGr}_i$,
$pr_{w_{i+1}\T t^{-1}} U_i\subset U_{i+1}$ and  $U_n=t^{-1}U_0$.
\end{dfn}

\begin{rem}\label{flgr}
We note that if $(U_i)_{i=0}^n\in \Fl^a(\mgl_n)$, then
\[
pr_{W\T 1} tU_0 = pr_{w_1\T 1}\dots pr_{w_n\T 1} tU_0\subset U_0.
\]
Indeed,
\[
pr_{W\T t^{-1}} tU_n = pr_{w_1\T t^{-1}}\dots pr_{w_n\T t^{-1}}(U_0)\subset U_n.
\]
Since $tU_n=U_0$ we obtain $pr_{W\T 1} tU_0\subset U_0$.
\end{rem}

Now let us consider the case of the affine Lie algebras. 
The standard PBW filtration $F_\bullet$ on the universal enveloping algebra
$\U(\widehat{\fg}^-)$ induces a filtration on the module $L_{\la,k}$. The associated graded space
$L^a_{\la,k}$ is a representation of the abelian Lie algebra $\widehat{\fg}^{-,a}$ (with the underlying vector space $\widehat{\fg}^{-}$).
\begin{lem}\label{nilp}
For any $x\in \widehat{\fg}^{-,a}$ there exists $N$ such that $x^N v_{\la,k}=0$ in $L_{\la,k}^a$.
\end{lem}
\begin{proof}
Recall that the algebra $\widehat{\fg}^{-,a}$ is abelian.
It suffices to prove the lemma for $x=r\T t^{-l}$ for some $r\in\fg$, $l>0$.
Since $r$ can be included into an $\msl_2$-triple, we may assume $\fg=\msl_2$.
Let $e,h,f$ be the standard basis of $\msl_2$. We know that
for any $i$ there exists $M$ such that $(f\T t^{-i})^Mv_{\la,k}=0$. Also we have the adjoint action of
$\msl_2=\msl_2\T 1$ on the symmetric algebra of $\msl_2$, generating new relations.
Hence, it suffices to prove that the algebra Sym$(\msl_2)/(U(\msl_2)f^M)$ is finite-dimensional.
But this algebra is known to be finite-dimensional by \cite{F3}, Corollary 4.3.
\end{proof}

\begin{rem}
We note that this lemma does not hold in the non-degenerate situation.
\end{rem}

Let $\widehat{G}^{-,a}=\exp(\widehat{\fg}^{-,a})$ be the Lie group of the Lie algebra $\widehat{\fg}^{-,a}$. This group is isomorphic to the
sum of (an infinite number of) copies of the groups $\bG_a$ (one copy for each negative root of $\gh$).
Let us now consider the projective space $\bP(L^a_{\la,k})$. For a vector $v\in L^a_{\la,k}$ let $[v]\in \bP(L^a_{\la,k})$ be
the corresponding line in the projectivization.

\begin{dfn}\label{dafv}
The degenerate affine flag variety $\Fl^a_{\la,k}$ is the closure of the orbit $\widehat{G}^{-,a}[v_{\la,k}]$
inside $\bP(L^a_{\la,k})$.
\end{dfn}

\begin{rem}
The orbit $\widehat{G}^{-,a}[v_{\la,k}]$ makes sense, since the action of the group $\widehat{G}^{-,a}$ on
$\bP(L^a_{\la,k})$ is well defined. In fact, thanks to Lemma \ref{nilp}, all the operators from the Lie algebra 
$\widehat{\fg}^{-,a}$ of the Lie group $\widehat{G}^{-,a}$ act as nilpotent operators and hence can be exponentiated,
giving rise to the $\widehat{G}^{-,a}$-action.
\end{rem}

\begin{rem}
We do not describe the degeneration procedure in this paper. Note however that for finite-dimensional Lie algebras
of type $A_n$ the closure of the orbit of the group  $\widehat{G}^{-,a}$ (of its finite-dimensional analogue, to be precise)
is indeed a (flat) degeneration of the classical flag variety (see \cite{F1}, Proposition 5.8). 
\end{rem}

\section{Degenerate affine Grassmannians of type $A$}
\subsection{Main theorems}
\begin{dfn}
The degenerate affine Grassmannian ${\rm Gr}^a(\mgl_n)$ of type $\mgl_n$ is the subvariety of the Sato Grassmannian
${\rm SGr}_0$ consisting of subspaces $U$ such that $pr_{W\T 1} tU\subset U$, where
$pr_{W\T 1}$ is the projection along $W\T 1$ to the span of $W\T t^i, i\ne 0$.
\end{dfn}

We also define finite-dimensional approximations of the varieties ${\rm Gr}^a(\mgl_n)$.
\begin{dfn}
${\rm Gr}^a_N(\mgl_n)$ is the subvariety of ${\rm Gr}^a(\mgl_n)$ consisting
of $U$ such that
\[
W\T t^N\bC[t] \subset U\subset W\T t^{-N}\bC[t].
\]
\end{dfn}

\begin{rem}\label{B}
Let $S_{N,n}=\frac{W\T t^{-N}\bC[t]}{W\T t^N\bC[t]}$. Then ${\rm Gr}_N^a(\mgl_n)$ is isomorphic to the variety of $Nn$-dimensional
subspaces $U\subset S_{N,n}$ such that $pr_{W\T 1} tU\subset U$.
\end{rem}

We note that ${\rm Gr}^a(\mgl_n)$ is naturally the inductive limit of its finite dimensional pieces ${\rm Gr}^a_N(\mgl_n)$.
We will prove the following theorems:

\begin{thm}\label{fdim}
\begin{itemize}
\item ${\rm Gr}^a_N(\mgl_n)$ is an irreducible projective variety of dimension $Nn^2$.
\item ${\rm Gr}^a_N(\mgl_n)$ carries an action of an abelian unipotent group $\bG_a^{Nn^2}$ with an open dense orbit.
\item ${\rm Gr}^a_N(\mgl_n)$ is isomorphic to an affine Schubert variety for the group $\widehat{SL}_{2n}$.
\end{itemize}
\end{thm}

\begin{thm}\label{infdim}
\begin{itemize}
\item ${\rm Gr}^a(\mgl_n)$ carries an action of an infinite dimensional abelian unipotent group with an open dense orbit. The group is the inductive
limit of the groups $\bG_a^{Nn^2}$ from Theorem \ref{fdim}.
\item ${\rm Gr}^a(\mgl_n)$ is isomorphic to a semi-infinite orbit closure inside the affine Grassmannian for the group $\widehat{SL}_{2n}$.
\end{itemize}
\end{thm}

\begin{lem}\label{irr}
${\rm Gr}^a_N(\mgl_n)$ is an irreducible projective variety of dimension $Nn^2$.
\end{lem}
\begin{proof}
We prove the claim by constructing a resolution of singularities.  
We consider the space $S_{N,n}=\frac{W\T t^{-N}\bC[t]}{W\T t^N\bC[t]}$. The operator
$pr_{W\T 1} t$ naturally acts on $S_{N,n}$. In particular, $(pr_{W\T 1}t)^N=0$. Let $S_{N,n}(k)={\rm Im} (pr_{W\T 1})^k$
for $k=0,1,\dots,N-1$. In particular, $\dim S_{N,n}(k)=n(2N-2k)$.

Let $R_{N,n}$ be the variety of collections of vector spaces $(U_0,U_1,\dots,U_N)$ such that the following holds:
\begin{enumerate}
\item $U_i\subset S_{N,n}(k),\; \dim U_i=n(N-k)$,
\item $pr_{W\T 1}t U_i\subset U_{i+1},\; i=0,\dots, N-1$.
\end{enumerate}
One easily sees that $R_{N,n}$ is an $N$-floor tower of fibrations over a point, each fibration having fiber $Gr(n,2n)$.
In addition, $R_{N,n}$ surjects onto ${\rm Gr}^a_N(\mgl_n)$ (forgetting all $U_i$ with $i>0$) and this surjection is generically one-to-one.
In fact, let $R_{N,n}^0\subset R_{N,n}$ be the subvariety cut out by the condition $\dim (pr_{W\T 1})^k t U_0=n(N-k)$.
By definition of $R_{N,n}$, the restriction of the surjection $R_{N,n}\to {\rm Gr}^a_N(\mgl_n)$ to 
$R_{N,n}^0$ is one-to-one. Now easily sees that $R_{N,n}^0$ is an open dense part of $R_{N,n}$.
This proves the lemma.  
\end{proof}

\begin{rem}
We generalize the claim of Lemma \ref{irr} in Example \ref{XGr} and Corollary \ref{geocon}.
\end{rem}

In order to prove the remaining statements of Theorem \ref{fdim} and Theorem \ref{infdim} we construct a
projective embedding of ${\rm Gr}^a(\mgl_n)$ and hence of its finite-dimensional pieces ${\rm Gr}^a_N(\mgl_n)$.

\subsection{Semi-infinite abelianization}
Let us decompose the Lie algebra $\mgl_\infty$ into four blocks
\[
\mgl_\infty=\mgl_\infty^{-,-}\oplus \mgl_\infty^{-,+}\oplus \mgl_\infty^{+,-}\oplus\mgl_\infty^{+,+},
\]
where $\mgl_\infty^{-,+}$ is spanned by the matrix units $E_{i,j}$ with $i\le 0$, $j>0$,
$\mgl_\infty^{+,-}$ is spanned by the matrix units $E_{i,j}$ with $i>0$, $j\le 0$ and similarly for the two other summands.
For example, $\mgl_\infty^{+,-}$ is spanned by the operators mapping $v_{\le 0}$ to $v_{>0}$.
\begin{rem}
The summand $\mgl_\infty^{+,-}$ is abelian.
\end{rem}

Let $p:\mgl_\infty\to \mgl_\infty^{+,-}$ be the projection along the three other summands. Recall the embedding
$\mgl_n\T\bC[t,t^{-1}]\subset\mgl_\infty$. Combining this embedding with $p$ we obtain a map
\[
\mgl_n\T t^{-1}\bC[t^{-1}]\to \mgl_\infty^{+,-}.
\]
We define the space
\[
L_0(\mgl^a_n)=U(p(\mgl_n\T t^{-1}\bC[t^{-1}]))|0\ket.
\]
Since $\mgl_\infty^{+,-}$ is abelian, $L_0(\mgl^a_n)$ is a cyclic representation of the abelian
Lie algebra $\mgl^a_n\T t^{-1}\bC[t^{-1}]$.

\begin{example}\label{h(z)}
Let us consider the Lie algebra ${\widehat{\mgl_1}}$. This is nothing but the Heisenberg algebra. Let  $F^{(0)}$ be the
Fock module. Let us denote by $h_i$, $i\in\bZ$ the basis of ${\widehat{\mgl_1}}$. Then the elements $h_i$, $i<0$ are
represented by the formula $h_i=\sum_{k\in\bZ} E_{k-i,k}$.

The projection $p$ to $\mgl_\infty^{+,-}$ defines a representation of the abelian Lie algebra spanned by
$h_i$, $i<0$ on $F^{(0)}$.
This  representation is no longer irreducible. In particular, the subspace generated from the highest weight vector is
defined by the relation $(h_{-1}+zh_{-2}+\dots)^2=0$. Here $z$ is a variable and the relation above 
means that the coefficients of $z^k$ vanish for all $k\ge 0$.   
\end{example}

Let us consider the affine Lie algebra $\widehat{\msl}_{2n}$ and its vacuum representation $L_0(\msl_{2n})$ with
the highest weight vector $l_0$. Let us embed $\mgl_n^a$ into $\msl_{2n}$ as the unipotent radical corresponding to the weight $\omega_n$.
Explicitly, we consider the abelian subalgebra $\fa$ in $\msl_{2n}$, spanned by the matrix units $E_{i,j}$, $n+1\le i\le 2n$, $1\le j\le n$.
Clearly, $\fa\simeq \mgl_n^a$.

\begin{lem}\label{2n}
The identification $\fa\simeq \mgl_n^a$ induces an isomorphism between $L_0(\mgl_n^a)$ and the subspace
$U(\fa\T t^{-1}\bC[t^{-1}])l_0$ of $L_0(\msl_{2n})$.
\end{lem}
\begin{proof}
Consider the space $R$ with basis $r_1,\dots,r_{2n}$. We write
$R=R^<\oplus R^>$, where $R^<$ is spanned by $r_1,\dots,r_n$ and $R^>$ is spanned by $r_{n+1},\dots,r_{2n}$.
Then $L_0(\msl_{2n})$ sits inside
$\Lambda^{\infty/2}(R\T\bC[t,t^{-1}])$ with $l_0$ being the wedge product
\[
l_0=\bigwedge_{i\ge 0} \left(r_1\T t^i\wedge\dots\wedge r_{2n}\T t^i\right).
\]
Let us write $R\T\bC[t,t^{-1}]$ as a direct sum of four subspaces $R^{<,+}$, $R^{<,-}$,
$R^{>,+}$ and $R^{>,-}$, where
\begin{gather*}
R^{<,+}=R^<\T \bC[t],\ R^{<,-}=R^<\T t^{-1}\bC[t^{-1}]
\end{gather*}
and similarly for $R^{>,+}$ and $R^{>,-}$. For example, $l_0$ is the wedge product of the ``top'' wedge powers of
$R^{<,+}$ and of $R^{>,+}$.

Now let us look at the subspace $U(\fa\T t^{-1}\bC[t^{-1}])l_0$. We take $a\in \fa$ and consider the vector $(a\T t^i) l_0$.
Clearly, the only nontrival terms showing up come from the action of $a\T t^i$ 
sending $R^{<,+}$ to $R^{>,-}$.
Recall the space $W$ with the basis $w_1,\dots,w_n$ used to construct the wedge representation for $\widehat{\mgl}_n$.
Let us embed $W\T \bC[t,t^{-1}]$ into $R\T\bC[t,t^{-1}]$ as follows:
\[
w_j\T t^i\mapsto r_j\T t^i,\ i\ge 0,\ w_j\T t^i\mapsto r_{j+n}\T t^i,\ i< 0.
\]
Thus the image of this embedding coincides with $R^{<,+}\oplus R^{>,-}$.
Now it remains to note that the operators of the form $p(\mgl_n^a)$ map
$W\T \bC[t]$ to $W\T t^{-1}\bC[t^{-1}]$.
\end{proof}

Let us consider the Lie group $GL_n^a(t^{-1}\bC[t^{-1}])=\exp(\mgl_n^a\T t^{-1}\bC[t^{-1}])$.
The group $GL_n^a(t^{-1}\bC[t^{-1}])$ is isomorphic to the direct sum
of an infinite number of copies of the additive group $\bG_a$ of the base field.
Consider the action of $GL_n^a(t^{-1}\bC[t^{-1}])$ on the projectivization $\bP(L_0(\mgl_n^a))$ and
(temporarily) denote by $G(n)$ the closure of the orbit through the line containing the highest weight vector.
Our goal is to identify $G(n)$ with ${\rm Gr}^a(\mgl_n)$.
Let us define a finite-dimensional approximation of $G(n)$ as follows:
\[
G_N(n)=\overline{\exp\left(\bigoplus_{i=1}^N \mgl_n^a\T t^{-i}\right)\cdot\bC|0\ket}.
\]

\begin{lem}\label{GGr}
Let $U\in G_N(n)$. Then $pr_{W\T 1}(tU)\subset U$ and $W\T t^N\bC[t] \subset U\subset W\T t^{-N}\bC[t]$.
\end{lem}
\begin{proof}
Take an element $g=\exp(\sum_{i=1}^N x_i\T t^{-i})$, $x_i\in\mgl_n$. Then the space corresponding to the line
$g\cdot\bC|0\ket$ is spanned by the vectors
\[
w_j\T t^k + x_iw_j\T t^{k-i},\ j=1,\dots,n,\ i=1,\dots,N,\ k=0,\dots,N-1,\ k-i<0.
\]
Now the claim is clear.
\end{proof}

\begin{prop}\label{group}
$G_N(n)\simeq{\rm Gr}_N^a(\mgl_n)$.
\end{prop}
\begin{proof}
We know that both $G_N(n)$ and ${\rm Gr}_N^a(\mgl_n)$ are irreducible. According to Lemma \ref{GGr}, $G_N(n)$ sits inside
${\rm Gr}^a(\mgl_n)$. Also note that the open orbit
\[
\exp\left(\bigoplus_{i=1}^N \mgl_n^a\T t^{-i}\right)\cdot\bC|0\ket
\]
 consists of all $U\in \frac{W\T t^{-N}\bC[t]}{W\T t^N\bC[t]}$ such that the Pl\"ucker coordinate of $U$, corresponding to
the set of vectors $w_j\T t^i$, $i\le 0$ does not vanish. Hence the open parts of $G_N(n)$ and
${\rm Gr}^a(\mgl_n)$ coincide.
\end{proof}

\noindent{\bf Proof of Theorem \ref{fdim}} The first claim follows from Lemma \ref{irr} and the second claim
follows from Proposition \ref{group}. The proof of the last claim is postponed until the end of the next section
(see Corollary \ref{aSv}).  \qed

\noindent{\bf Proof of Theorem \ref{infdim}}
Since $G_N(n)\simeq{\rm Gr}_N^a(\mgl_n)$, we obtain $G(n)\simeq{\rm Gr}^a(\mgl_n)$. 
By Lemma \ref{2n} we obtain that ${\rm Gr}^a(\mgl_n)$ is naturally embedded into the classical affine Grassmannian for the group $\widehat{SL}_{2n}$.
The image of this embedding is the closure of the orbit of the abelian unipotent group $\exp(\fa\T t^{-1}\bC[t^{-1}])$
through the highest weight line. \qed

\subsection{$\msl_n$ case}\label{tilde}
Let $L_0(\msl_n)$ be the basic level one $\msl_n$ module with highest weight vector $l_0$.
Recall the PBW graded version $L_0^a(\msl_n)$; in particular, $L_0^a(\msl_n)=U(\msl_n^a\T t^{-1}\bC[t^{-1}])l_0$.
We define
\[
\tilde L_0(\msl_n)=U(p(\msl_n\T t^{-1}\bC[t^{-1}]))|0\ket.
\]

\begin{conj}\label{T}
The $\fn^{-,a}\T t^{-1}\bC[t^{-1}]$-modules $\tilde L_0(\msl_n)$ and $L_0^a(\msl_n)$ are isomorphic.
\end{conj}

We prove the following lemma.
\begin{lem}
There exists a surjection $L_0^a(\msl_n)\to \tilde L_0(\msl_n)$ of $\msl_n^a\T t^{-1}\bC[t^{-1}]$ modules.
\end{lem}
\begin{proof}
We need to show that any relation which holds in $L_0^a(\msl_n)$ is true in $\tilde L_0(\msl_n)$. We consider a grading
on the semi-infinite wedge space $\Lambda^{\infty/2}(V)$. The $s$-th graded component is spanned by vectors
$v_{i_1}\wedge v_{i_2}\wedge\dots$ such that the number of positive indices among $i_l$ is equal to $s$. For example,
the zeroth component is spanned by $|0\ket$. Then any $x\in \mgl_\infty^{+,-}$ increases this grading by one.
The operators from the three remaining summands either decrease the grading or preserve it.

Now assume that we have a relation in $L_0^a(\msl_n)$. To make it explicit, let us fix a basis $x_1,x_2,\dots$ in
$\fn^{-}\T t^{-1}\bC[t^{-1}]$. Then any relation can be represented by a homogeneous total degree $N$ polynomial $p(x_1,x_2,\dots)$, which vanishes
in  $L_0^a$. This means that we have an equality in $L_0$:
\[
p(x_1,x_2,\dots) |0\ket = q(x_1,x_2,\dots) |0\ket \text{ for some } q \text{ of degree less than } N.
\]
This equality implies that $p(x_1,x_2,\dots) |0\ket$ vanishes in $\tilde L_0(\msl_n)$, since this is exactly the degree $N$
component of $p(x_1,x_2,\dots) |0\ket$, considered as a vector in $L_0(\msl_n)$.
\end{proof}

\begin{rem}
To complete the proof of Conjecture \ref{T} it thus suffices to show that the character of $L_0^a(\msl_n)$
is equal to the character of $\tilde{L}_0(\msl_n)$. Unfortunately, we are not able to do this at the moment.
\end{rem}

\subsection{Flatness}
In this section we discuss the infinite-dimensional analogue of the flatness of the degeneration of the $\msl_n$
flag varieties.
We consider the space $W\T \bC[t,t^{-1}]$, $W=\bC^n$. Let $U\in{\rm SGr}_0$ be a subspace with $tU\subset U$. We start with the
open cell containing the base point $U_0=W\T \bC[t]$. The coordinates in this cell are given by the collection
of linear mappings $(A_{k,i})$, $k\ge 1$, $i\ge 0$,  $A_{k,i}\in{\rm End} (W)$, such that for a fixed $i$,
the operators $A_{k,i}$ vanish for $k$ large enough.
The subspace $U$ corresponding to a collection $(A_{k,i})$ is defined as the linear span of the vectors
\begin{equation}\label{wti}
w\T t^i + A_{1,i}w\T t^{-1}+A_{2,i}w\T t^{-2}+\dots.
\end{equation}
Now we introduce the operator $\al_\hbar\in ({\rm End} W\T \bC[t,t^{-1}])$ defined by
\[
\al_\hbar (w\T t^i)=
\begin{cases}
w\T t^i,\text{ if } i\ne 0,\\
\hbar w\T t^0,\text{ if } i=0.
\end{cases}
\]

We consider the ind-subscheme ${\bf Gr}\subset{\rm SGr}_0\times\BA^1$ in
the product of the Sato Grassmannian and the affine line with coordinate
$\hbar$ cut out by the equations $\al_\hbar tU\subset U$.
We note that if $\hbar\ne 0$ then the fiber of ${\bf Gr}$ over $\hbar$
is isomorphic to the affine Grassmannian for $\mgl_n$.
The special fiber ${\bf Gr}_0$ is the degenerate affine Grassmannian.

\begin{prop}
\label{flatn}
The morphism ${\bf Gr}\to\BA^1$ is flat.
\end{prop}

\begin{proof}
For each connected component ${\bf Gr}^n,\ n\in{\mathbb Z}$, of ${\bf Gr}$,
it suffices to produce a dense open affine subscheme $S^n\subset{\bf Gr}^n$
such that $\hbar\in\bC[S^n]$ is not a zero divisor. We will exhibit $S^0$;
the other connected components are taken care of similarly.
We define $S^0$ as the intersection of ${\bf Gr}^0$ with the open cell in
${\rm SGr}_0$ formed by the subspaces transversal to $W\otimes t^{-1}\bC[t^{-1}]$.

\begin{lem}
Let $U\in {\rm SGr}_0$ be defined by the collection of operators $(A_{k,i})$.
Then $(U,\hbar)\in{\bf Gr}$ if and only if
\begin{equation}\label{Aki}
A_{k,i}=A_{k-1,i+1}+\hbar A_{k-1,0}A_{1,i}
\end{equation}
for all $k,i\ge 1$.
\end{lem}
\begin{proof}
We need
\begin{equation}\label{tU}
w\T t^{i+1} +\hbar A_{1,i}w\T t^{0}+A_{2,i}w\T t^{-1}+\dots\in U.
\end{equation}
This means that the vector \eqref{tU} is a linear combination of the vectors \eqref{wti} starting with $w\T t^{i+1}$ and
$w\otimes t^0$ with coefficients 1 and $\hbar A_{1,i}$. 
This gives the desired system of equations.
\end{proof}

\begin{rem}
If $\hbar=0$ then equation \eqref{Aki} reduces to $A_{k,i}=A_{k-1,i+1}$ (block Hankel matrices).
\end{rem}

Now the ind-scheme $S^0$ is a union of the finite-type subschemes $S^0(N)$
cut out by the conditions $A_{k,i}=0$ if $k>N$ or $i\ge N$.
The relations \eqref{Aki} allow to express any operator $A_{k,i}$
with $k>0$ in terms of $A_{1,i}$. Hence the coordinate ring $\bC[S^0]$ is a
certain completion of the polynomial ring in $\hbar$ and the matrix
coefficients of $A_{1,0},A_{1,1},A_{1,2},\ldots$ The relations~\eqref{Aki} are
homogeneous if we set $\deg\hbar=0,\ \deg A_{k,i}=k+i$. If we had a relation
$\hbar f=0$ for $0\ne f\in\bC[S^0]$, it would imply $\hbar f^{(n)}=0$ for each
homogeneous component $f^{(n)}$ of $f$. Since each $f^{(n)}$ is an element of
the polynomial ring in $\hbar$ and the matrix
coefficients of $A_{1,0},A_{1,1},A_{1,2},\ldots$, we conclude $f^{(n)}=0$, hence
$f=0$. The proposition is proved.
\end{proof}

\begin{rem}
The special fiber ${\bf Gr}_0$ is a union of reduced
finite type schemes of growing dimensions. For $\hbar\ne0$, in the case
of $\mgl_1$, the fiber ${\bf Gr}_\hbar$ is a union of zero-dimensional schemes
with nilpotents. So the flatness of ${\bf Gr}$ over ${\mathbb A}^1$
holds only at the level of
inductive limit, and looks somewhat counterintuitive.
\end{rem}

\begin{rem}
Again let us consider the case of $\mgl_1$. Then for $\hbar\ne 0$ the
fiber lives in the projectivization of the Fock module $\bC[h_{-1},h_{-2},\dots]$
(the Fock module can be naturally identified with
the space of sections of the tautological line bundle). However, in the degenerate situation the
space of sections of the tautological line bundle is smaller (more precisely,
$\bC[h_{-1},h_{-2},\dots]/(h(z)^2)$, i.e. the ideal we quotient out is generated by all the coefficients
of the series $(\sum_{i>0} z^{i-1}h_{-i})^2$, see Example \ref{h(z)}).
 \end{rem}

\section{Quiver Grassmannians for (truncated) loop quivers}

\subsection{Basics}
For $N\geq 1$, let $A_N={\bC}[t]/(t^N)$ be the truncated polynomial ring, which we view as a finite-dimensional
and self-injective algebra over ${\bC}$ ($A_N$ being an indecomposable projective and injective module over itself).

For a finitely generated (possibly non-commutative) algebra $A$, a finite-dimensional $A$-module $M$ and an integer $k\leq\dim M$,
we denote by ${\rm Gr}^A_k(M)$ the Grassmannian of $k$-dimensional subrepresentations $U$ of $M$. This is a projective variety,
admitting a closed embedding into the Grassmannian ${\rm Gr}_k(M)$ of $k$-dimensional ${\bC}$-linear subspaces of $M$.

Let $W$ be an $m$-dimensional vector space, and consider $W\otimes A_N$ as an $A_N$-module, which is thus
projective and injective of dimension $mN$. Our aim is to study the following varieties and to relate them to degenerate
affine Grassmannians and to affine Schubert varieties:

\begin{definition} With notation as above, define
$$X^{(N)}_{k,m}={\rm Gr}_{k}^{A_N}(W\otimes A_N)$$
as the Grassmannian of $k$-dimensional $A_N$-subrepresentations of $W\otimes A_N$.
\end{definition}

\begin{example}\label{XGr}
$X^{(N)}_{Nn,2n}\simeq {\rm Gr}^a_N(\mgl_n)$ (see Remark \ref{B}).
\end{example}

Let us interpret this definition in linear algebra terms: consider the operator $\varphi$ on $W\otimes A_N$
which is given by $\varphi={\rm id}_W\otimes t$, where $t$ means multiplication by $t$ on $A_N$
(which is a regular nilpotent operator on $A_N$). Thus $\varphi$ is nilpotent and its Jordan canonical form consists of
$m$ nilpotent Jordan blocks of size $N$ each. The variety $X^{(N)}_{k,m}$ parametrizes $\varphi$-invariant $k$-dimensional
subspaces of $W\otimes A_N$, thus it is naturally a closed subvariety of ${\rm Gr}_k(W\otimes A_N)$.\\
Note that the group ${\rm GL}_m(A_N)\simeq {\rm GL}(W\otimes A_N)$ acts on $X^{(N)}_{k,m}$; this group is of dimension $mN$,
with reductive part ${\rm GL}_m({\bC})$ and unipotent radical $1+Mat_{m}(tA_N)$.\\
Our aim in the next subsection is to realize $X^{(N)}_{k,m}$ as a geometric quotient of a well-known variety by a (free) group action.
We recall the prototype for such quotient realizations. Let $V$ be a $k$-dimensional vector space.
Then the linear Grassmannian ${\rm Gr}_k(W\otimes A_N)$ can be viewed as the quotient of the set ${\rm Hom}^0(V,W\otimes A_N)$ of
injective linear maps from $V$ to $W\otimes A_N$ modulo the action of ${\rm GL}(V)$, thus
$${\rm Gr}_k(W\otimes A_N)\simeq {\rm Hom}^0(V,W\otimes A_N)/{\rm GL}(V).$$

\subsection{Interpretation as framed moduli}

We recall a result of classical invariant theory, see \cite{LBR}, Section 5.1:\\
For two vector spaces $V$ and $W$, consider the action of ${\rm GL}(V)$ on ${\rm End}(V)\times {\rm Hom}(V,W)$
by $g\cdot(\varphi,f)=(g\varphi g^{-1},fg^{-1})$. We call a pair $(\varphi,f)$ stable if $\bigcap_{i\geq 0}{\rm Ker}(f\varphi^i)=0$.
This is equivalent to the map $$\bigoplus_{i=0}^{\dim V-1}f\varphi^i:V\rightarrow W^{\dim V}$$ being injective.
Denote by $({\rm End}(V)\times {\rm Hom}(V,W))^{\rm st}$ the open subset of stable points; the action of ${\rm GL}(V)$ is free on this stable locus.

\begin{thm}
The set $({\rm End}(V)\times {\rm Hom}(V,W))^{\rm st}$ admits a geometric quotient by ${\rm GL}(V)$. It
embeds into the Grassmannian ${\rm Gr}_{\dim V}(W^{\dim V})$ of subspaces of $W^{\dim V}$ of dimension $\dim V$, by mapping the class of
$(\varphi,f)$ to the subspace ${\rm Im}(\bigoplus_{i=0}^{\dim V-1}f\varphi^i)$.
\end{thm}

Let $\mathcal{N}(V)$ be the closed subvariety of nilpotent operators in ${\rm End}(V)$, and let $\mathcal{N}^{(N)}(V)$ be
the set of operators $\varphi$ such that $\varphi^N=0$, for $N\geq 1$.

As restrictions of (geometric) quotients to invariant closed subvarieties are again (geometric) quotients, we can consider the
${\rm GL}(V)$-invariant subset $(\mathcal{N}^{(N)}(V)\times{\rm Hom}(V,W))^{\rm st}$ of $({\rm End}(V)\times {\rm Hom}(V,W))^{\rm st}$ and get,
after some reindexing and using $N$-nilpotency, the following result:

\begin{cor}
The set $(\mathcal{N}^{(N)}(V)\times {\rm Hom}(V,W))^{\rm st}$ admits a geometric quotient by ${\rm GL}(V)$.
This quotient embeds into the Grassmannian ${\rm Gr}_{\dim V}(W\otimes A_N)$ of $dim V$-dimensional subspaces of
$W[[t]]/(t^N)\simeq W\otimes A_N$ by mapping the class of $(\varphi,f)$ to the subspace ${\rm Im}(\sum_{i=0}^{N-1}f\varphi^{N-1-i}t^i)$.
\end{cor}

The main technical result of this section is the following (compare with \cite{L}, section 2):

\begin{lem} The image of the above embedding coincides with the subvariety $X^{(N)}_{\dim V,\dim W}$ of ${\rm Gr}_{\dim V}(W\otimes A_N)$.
\end{lem}

\begin{proof} For $\varphi$ and $f$ such that $\varphi^N=0$, the image of $\sum_if\varphi^{N-1-i}t^i$ is $t$-invariant, namely
$$t\sum_{i=0}^{N-1}f\varphi^{N-1-i}t^i=\sum_{i=0}^{N-1}f\varphi^{N-1-i}t^{i+1}=\sum_{i=0}^{N-1}f\varphi^{N-i}t^i=(\sum_{i=0}^{N-1}f\varphi^{N-1-i}t^i)\varphi,$$
and thus $$t\cdot{\rm Im}(\sum_if\varphi^{N-1-i}t^i)={\rm Im}(t\sum_if\varphi^{N-1-i}t^i)=$$
$$={\rm Im}((\sum_i f\varphi^{N-1-i}t^i)\varphi)\subset{\rm Im}(\sum_if\varphi^{N-1-i}t^i).$$
Conversely, let the image of an injective map $\sum_{i=0}^{N-1}f_it^i$ be $t$-invariant. This means that there exists an endomorphism $\varphi$ such that
$$\sum_{i=1}^{N}f_{i-1}t^i=t\sum_{i=0}^{N-1}f_it^i=(\sum_{i=0}^{N-1}f_it^i)\varphi=\sum_{i=0}^{N-1}f_i\varphi t^i.$$
Comparing coefficients, this is equivalent to
$$f_0\varphi=0,\; f_0=f_1\varphi,\, f_1=f_2\varphi,\,\ldots\, f_{N-2}=f_{N-1}\varphi,$$
thus $f_i=f_{N-1}\varphi^{N-1-i}$ for all $i=0,\ldots,N-1$, and $f_{N-1}\varphi^N=0$. The latter conditions means that
${\rm Im}(\varphi^{i+1})\subset{\rm Ker}(f_{N-1}\varphi^{N-1-i})$ for all $i$, thus
$${\rm Im}(\varphi^N)\subset \bigcap_{i=0}^{N-1}{\rm Ker}(f_{N-1}\varphi^{N-1-i})=\bigcap_{i=0}^{N-1}{\rm Ker}(f_i)=0$$
by injectivity of $\sum_if_it^i$, which means $\varphi^N=0$. Thus, the pair $(\varphi,f_{N-1})$ maps to the image of $\sum_if_it^i$.\end{proof}

We have thus proved:

\begin{cor}\label{maincor} The variety $X^{(N)}_{\dim V,\dim W}$ is isomorphic to the quotient $$(\mathcal{N}^{(N)}(V)\times{\rm Hom}(V,W))^{\rm st}/{\rm GL}(V).$$
\end{cor}

\subsection{Geometric consequences}

The above corollary allows us to easily derive various geometric properties of the varieties $X^{(N)}_{k,m}$:

\begin{cor} \label{geocon}
The variety $X^{(N)}_{k,m}$ is irreducible, normal, Cohen-Macaulay with rational singularities. It has dimension
$\dim\mathcal{N}^{(N)}({\bC}^k)+k(m-k)$. In particular, it has dimension $k(k-1)+k(m-k)=k(m-1)$ for $N\geq k$.
\end{cor}

\begin{proof} Every variety $\mathcal{N}^{(N)}({\bC}^k)$ is irreducible, normal, Cohen-Macaulay with rational singularities by \cite{KP}, Theorem 0.1, since it is the closure of a single conjugacy class.
These properties are preserved under passing to open subsets and geometric ${\rm GL}_k({\bC})$-quotients.
Namely, for irreducibility this is clear, whereas normality and rational singularities are preserved under
arbitrary quotients (the latter by Boutot's theorem (\cite{B}, Corollaire)). For the Cohen-Macaulay property, we use the fact that a geometric
${\rm GL}_k({\bC})$-quotient is a principal bundle under a special group, and thus Zariski locally trivial (see \cite{Se}, Theoreme 2).
Finally, the dimension formula follows by a direct calculation.\end{proof}

The natural sequence of embeddings of the varieties $\mathcal{N}^{(N)}(V)$, stabilizing in $\mathcal{N}^{(\dim V)}(V)=\mathcal{N}(V)$,
induces a chain of embeddings $$X_{k,m}^{(1)}\subset X_{k,m}^{(2)}\subset \ldots\subset X_{k,m}^{(k)}=X_{k,m}^{(k+1)}=\ldots,$$
whose limit we define as $X_{k,m}=X_{k,m}^{(N)}$ for $N\geq k$, which is the quotient of stable pairs $(\varphi,f)$ for $\varphi$
an arbitrary nilpotent operator by the action of ${\rm GL}_k({\bC})$.

\subsection{Examples}

We give some examples of the varieties $X_{k,m}$. First, it is easy to see that $X_{1,m}\simeq{\bf P}^{m-1}$.
Second, let us consider the case $k=2$. By the above, we thus consider the set of $2m\times 2$-matrices of rank $2$ of the form
$$\left[\begin{array}{l}AB\\ A\end{array}\right]$$
for an $m\times 2$-matrix $A$ and a nilpotent $2\times 2$-matrix $B$, up to the ${\rm GL}_2$-action on columns.
We can embed this variety into projective space via the Pl\"ucker embedding. Namely, we choose homogeneous coordinates
$x_{i,j}$ for $1\leq i<j\leq 2m$ and map the above matrix to the collection of its $2\times 2$-minors $T_{i,j}$.
A priori these obey the Pl\"ucker relations
$$T_{i,k}T_{j,l}=T_{i,j}T_{k,l}+T_{i,l}T_{j,k}$$
for all $1\leq i<j<k<l\leq 2m$. Since $\det(B)=0$, we have $T_{i,j}=0$ for $1\leq i<j\leq m$. Since ${\rm tr}(B)=0$,
we have $T_{i,j+m}=T_{j,i+m}$ for all $1\leq i,j\leq m$. We conclude that $X_{2,m}$ can be realized as the set of points
in projective space with coordinates $x_{i,j}$ for $1\leq i<j\leq 2m$ subject to the relations
$$x_{i,k}x_{j,l}=x_{i,j}x_{k,l}+x_{i,l}x_{j,k}\mbox{ for }1\leq i<j<k<l\leq 2m,$$
$$x_{i,j}=0\mbox{ for }1\leq i<j\leq m,$$
$$x_{i,j+m}=x_{j,i+m}\mbox{ for }1\leq i,j\leq m.$$
For example, in case $m=2$ we can eliminate the variables $x_{1,2}$ and $x_{2,3}$ and realize $X_{2,2}$
as the singular surface in ${\bf P}^3$ with coordinates $x_{1,3},x_{1,4},x_{2,4},x_{3,4}$ and defining equation $x_{1,3}x_{2,4}=x_{1,4}^2$
(then $(0:0:0:1)$ is an isolated singularity).\\
For general $m$, we can eliminate the variables $x_{i,j}$ for $1\leq i<j\leq m$ and $x_{j,i+m}$ for $1\leq i<j\leq m$.
We can rename the remaining variables as $v_{i,j}=x_{i+m,j+m}$ for $1\leq i<j\leq m$ and $w_{i,j}=x_{i,j+m}$ for $1\leq i,j\leq m$ and
rewrite the above relations in these terms. Computer experiments with $m\leq 5$ suggest the following:

\begin{conj} The variety $X_{2,m}$ is isomorphic to the closed subvariety of projective space with coordinates $v_{i,j}$
for $1\leq i<j\leq m$ and $w_{i,j}$ for $1\leq i\leq j\leq m$ given by the following equations:
\begin{enumerate}
\item the symmetric matrix $W=(w_{i,j})$ has rank one,
\item $WZ=0$ for the matrix $Z$ with $m$ rows, with columns indexed by tuples $(j,k,l)$ for $1\leq j<k<l\leq m$, and with entries
$$Z_{i,(j,k,l)}=\left\{\begin{array}{lll}v_{k,l}&,&i=j,\\
-v_{j,l}&,&i=k,\\
v_{j,k}&,&i=l,\\
0&,&i\not=j,k,l\end{array}\right.$$
\item the Pl\"ucker relations for the $v_{i,j}$.
\end{enumerate}
\end{conj}

\subsection{Orbit structure}

Now we consider the action of ${\rm GL}_m(A_N)$ on $X_{k,m}^{(N)}$ and determine the orbit structure. The group ${\rm GL}_m(A_N)$
embeds into the group ${\rm GL}(W\otimes A_N)$ as the subgroup of automorphisms commuting with ${\rm id}_W\otimes t$; every such automorphism
can be written uniquely as
$$\sum_{i=0}^{N-1}\psi_it^i$$ for $\psi_0\in{\rm GL}(W)$, $\psi_i\in{\rm End}(W)$ for $i=1,\ldots,N-1$,
where each summand acts on $W\otimes A_N$ by applying the endomorphism $\psi_i$ in the $W$-component and multiplying by $t^i$ in the $A_N$-component.
In particular, such an element acts on a linear map $\sum_if_it^i$ from $V$ to $W\otimes A_N$ by
$$(\sum_i\psi_it^i)(\sum_if_it^i)=\sum_i(\sum_{i'+i''=i}\psi_{i'}f_{i''})t^i.$$
This allows us to conclude:

\begin{lem}\label{orb}
Under the isomorphism $$X^{(N)}_{\dim V,\dim W}\simeq(\mathcal{N}^{(N)}(V)\times{\rm Hom}(V,W))^{\rm st}/{\rm GL}(V),$$
the action of ${\rm GL}_{\dim W}(A_N)$ on $X_{\dim V,\dim W}^{(N)}$ translates to
$$(\sum_i\psi_it^i)\cdot\overline{(\varphi,f)}=\overline{(\varphi,\sum_i\psi_if\varphi^i)}$$
for $\sum_i\psi_it^i\in{\rm GL}_{\dim W}(A_N)$, $\varphi\in\mathcal{N}^{(N)}(V)$ and $f\in{\rm Hom}(V,W)$.
\end{lem}

\begin{proof} Compute the action of $\sum_i\psi_it^i$ on $\sum_if\varphi^{N-1-i}t^i$ as above and compare the $t^{N-1}$-coefficients.\end{proof}

To parametrize the orbits of ${\rm GL}_m(A_N)$ in $X_{k,m}^{(N)}$, we use representation theory of the algebra $A_N$
and some of the methods developed in \cite{CFR1}. More precisely, we will use the following facts:

\begin{itemize}
\item The indecomposable representations of $A_N$ are (up to isomorphism) the $U_i={\bC}[t]/(t^i)$ for $i=1,\ldots,N$.
The representation $U_i$ has dimension $i$ and socle $U_1$. We have $\dim{\rm Hom}_{A_N}(U_i,U_j)=\min(i,j)$. In particular,
$A_N$ admits only finitely many isomorphism classes of representations of fixed dimension.
\item Two subrepresentations $R,R'$ of an injective representation $I$ of a finite-dimensional algebra
$A$ are conjugate under ${\rm Aut}(I)$ if and only if they are isomorphic.
\item A representation $R$ embeds into an injective representation $I$ of a finite-dimensional algebra $A$
if and only if the socles embed, that is, ${\rm soc}(R)$ embeds into ${\rm soc}(I)$.
\item If an $A$-representation $M$ admits only finitely many isomorphism classes of subrepresentations of a given
dimension $k$, the Grassmannian of subrepresentations ${\rm Gr}_k^A(M)$ is stratified into locally closed subsets
$\mathcal{S}_{[R]}$ consisting of subrepresentations in a fixed isomorphism class $[R]$; we have
$\dim\mathcal{S}_{[R]}=\dim{\rm Hom}(R,M)-\dim{\rm End}(R)$.
\end{itemize}

Combining the above statements, we have:

\begin{prop} Suppose $A$ is an algebra of finite representation type, and let $I$ be an injective representation of $A$.
Then the ${\rm Aut}(I)$-orbits $\mathcal{O}_{[R]}$ in ${\rm Gr}_k(I)$ are parametrized by the isomorphism classes $[R]$
such that $R$ is $k$-dimensional and ${\rm soc}(R)\subset{\rm soc}(I)$. The orbit $\mathcal{O}_{[R]}$ has dimension
$\dim{\rm Hom}(R,I)-\dim{\rm End}(R)$.
\end{prop}

This proposition applies to our setting since $A=A_N$ is of finite representation type and $I=U_N^m$ is injective. Every $U_i$ has simple socle,
thus a representation embeds into $U_N^m$ if and only if it has at most $m$ indecomposable direct summands. A $k$-dimensional such
representation can be written as $U_\lambda=U_{\lambda_1}\oplus\ldots\oplus U_{\lambda_m}$ for a partition
$N\geq\lambda_1\geq\ldots\geq\lambda_m\geq 0$ with  $\sum_i\lambda_i=k$ (note the parts are allowed to be zero).\\

\begin{cor} The ${\rm GL}_m(A_N)$-orbits $\mathcal{O}_\lambda$ in $X_{k,m}^{(N)}$ are parametrized by partitions
$\lambda$ of $k$ of length $m$ with parts at most $N$. We have $$\dim\mathcal{O}_\lambda=mk-\sum_i(\lambda_i')^2,$$
where $\lambda'$ denotes the conjugate partition of $\lambda$.
\end{cor}
\begin{proof}
We know that the dimension of $\dim\mathcal{O}_\lambda$ is equal to 
$$\dim{\rm Hom}(U_\lambda,U_N^m)-\dim{\rm End}(U_\lambda).$$
We note that $\dim{\rm Hom}(U_{\lambda_i},U_N)=\lambda_i$ and 
$\dim{\rm Hom}(U_{\lambda_i},U_{\lambda_j})=\lambda_{{\rm max}(i,j)}$. Hence
$\dim\mathcal{O}_\lambda=(m+1)k-2\sum_{i=1}^m i\lambda_i$. Now using \cite[I.(1,6)]{Mac}, we arrive at the desired formula.
\end{proof}

Let us determine the closure relation of these orbits. We recall some facts on orderings on partitions from e.g. \cite{Mac}:

\begin{itemize}
\item We write $\lambda\geq\mu$ (the so-called dominance ordering on partitions) if $\sum_{j\leq i}\lambda_j\geq\sum_{j\leq i}\mu_j$ for all $i$.
\item We have $\lambda>\mu$ minimally if and only if there exist entries $i<j$ such that $\mu_i=\lambda_i-1$, $\mu_j=\lambda_j+1$ and
$\mu_k=\lambda_k$ for all $k\not=i,j$.
\item We have $\lambda\geq\mu$ if and only if for the conjugate partitions one has $\mu'\geq\lambda'$.
\end{itemize}

Consider partitions $\lambda$, $\mu$ as in the corollary.

\begin{thm}\label{adherence} The closure of $\mathcal{O}_\lambda$ contains $\mathcal{O}_\mu$ if and only if $\lambda\geq\mu$.
\end{thm}

\begin{proof} Suppose $\lambda\geq\mu$. To prove that $\mathcal{O}_\mu$ is contained in the closure of $\mathcal{O}_\lambda$,
it suffices to do this in the case where $\lambda\geq\mu$ minimally, thus $\mu$ differs from $\lambda$ only in two positions $i<j$ as above.
We can then reduce to the case $m=2$, thus we want to prove that the closure of the orbit corresponding to the partition $(\lambda_1,\lambda_2)$
contains the orbit corresponding to $(\lambda_1-1,\lambda_2+1)$ (in particular, $\lambda_1\ge\lambda_2+2$). 
We consider the following family $U_z$ for $z\in{\bC}$ of
$t$-invariant subspaces of $A_N^2$: the subspace $U_z$ is generated by
$(zt^{N-\lambda_1},t^{N-\lambda_1+1})$ and $(t^{N-\lambda_2-1},zt^{N-\lambda_2})$.
Then $U_z$ belongs to $\mathcal{O}_{(\lambda_1,\lambda_2)}$ for $z\not=0,1,-1$ and to $\mathcal{O}_{(\lambda_1-1,\lambda_2+1)}$ for $z=0$.
(The $z=0$ case is clear; if $z\ne 0$, then $U_z\simeq U_{\lambda_1}\oplus U_{\lambda_2}$ if and only if 
the vectors $t^{\lambda_1-\lambda_2-1}(zt^{N-\lambda_1},t^{N-\lambda_1+1})$ and $(t^{N-\lambda_2-1},zt^{N-\lambda_2})$ are
linearly independent, which is equivalent to $z\ne \pm 1$).\\
To prove the converse, we define for a sequence $k_*=(k_1,\ldots,k_N)$ the subset $$C(k_*)=\{U\in X_{k,m}^{(N)}\, :\, \dim t^iU\leq k_{i+1}\mbox{ for } i=0,\ldots,N-1\}.$$
This is a closed subset since the dimension inequalities can be interpreted as rank conditions. For $U\in\mathcal{O}_\lambda$, we have
$$\dim t^iU=\sum_{j:\lambda_j\geq i}(\lambda_j-i).$$
Thus, defining $k_*(\lambda)$ by $$k_i(\lambda)=\sum_{j:\lambda_j\geq i-1}(\lambda_j-i+1)=\sum_{j\geq i}\lambda_j',$$
we have $\mathcal{O}_\lambda\subset C(k_*(\lambda))$.  This immediately implies $\overline{\mathcal{O}_\lambda}\subset C(k_*(\lambda))$.
Since we already know that $\overline{\mathcal{O}_\lambda}$ contains all $\mathcal{O}_\mu$ for $\lambda\geq\mu$,
the claim follows once we prove that $C(k_*(\lambda))$ is contained in the union of the $\mathcal{O}_\mu$ for $\lambda\geq\mu$.
So let us assume that $\mathcal{O}_\mu$ is contained in $C(k_*(\lambda))$. Then we know, by the definitions, that
$$\sum_{j:\mu_j\geq i}(\mu_j-i)\leq\sum_{j:\lambda_j\geq i}(\lambda_j-i)$$
for all $i=0,\ldots,N-1$; 
this can be rewritten as $\mu'\geq\lambda'$, thus $\lambda\geq\mu$. The theorem is proved.\end{proof}

In particular, we can determine the open orbit of ${\rm GL}_m(A_N)$ in $X_{k,m}^{(N)}$; it corresponds
to the partition $(\underbrace{N,\ldots,N}_{s},r,0,\ldots,0)$,  where $k=sN+r$ for $0\leq r<N$.

\subsection{Desingularization}\label{resolution}

For $N,k,m$ as before and a sequence $(k_1,\ldots,k_N)$ as in the proof of Theorem \ref{adherence} (in particular, $k_1=k$),
we consider the variety $Y(k_*)$ of tuples $$(U_1,\ldots,U_{N})\in\prod_{i=0}^{N-1}{\rm Gr}_{k_{i+1}}(W\otimes t^iA_N)$$
such that
\begin{enumerate}
\item $U_1\supset U_2\supset\ldots\supset U_{N}$,
\item $t(U_i)\subset U_{i+1}$ for $i=1,\ldots,N-1$
\end{enumerate}

(as before, we abbreviate the map ${\rm id}_W\otimes t$ simply by $t$).
This is a closed subvariety of the product $\prod_{i=0}^{N-1}{\rm Gr}_{k_{i+1}}(W\otimes t^iA_N)$ and thus projective.

\begin{rem} We can view the variety $Y(k_*)$ as a quiver Grassmannian as follows: consider
the quiver $\Gamma_N$ with $N$ vertices $v_1,\ldots,v_{N}$ and $2(N-1)$ arrows
$\alpha_i:v_i\rightarrow v_{i+1}$ for $i=1,\ldots,N-1$ and $\beta_i:v_{i+1}\rightarrow v_i$ for $i=1,\ldots,N-1$. We consider the admissible ideal $I$
in ${\bC}\Gamma_N$ generated by the elements
$$\alpha_i\beta_i-\beta_{i+1}\alpha_{i+1},\, i=1,\ldots,N-1$$
(where for $i=N-1$, the relation has to be read as $\alpha_{N-1}\beta_{N-1}=0$, i.e. we formally define $\alpha_{N}$ and $\beta_{N}$ as zero).
The algebra ${\bC}\Gamma_N/I$ is then isomorphic to the Auslander algebra of $A_N$.

We construct a specific representation ${\bf W}$ of ${\bC}\Gamma_N/I$: we define ${\bf W}_{v_i}=W\otimes t^iA_N$, ${\bf W}_{\alpha_i}=t$
(multiplication by $t$ in the second component) and ${\bf W}_{\beta_i}=\iota$, the inclusion map. We see that this
representation of $\Gamma_N$ indeed satisfies the relations in $I$. Thus we can interpret $Y(k_*)$ as ${\rm Gr}_{k_*}^{{\bC}\Gamma_N/I}({\bf W})$.
\end{rem}

\begin{prop} For every partition $\lambda$ as above, projection to ${\rm Gr}_{k}(W\otimes A_N)$ induces a desingularization map
$$\pi_\lambda:Y(k_*(\lambda))\rightarrow\overline{\mathcal{O}_\lambda}.$$
\end{prop}

\begin{proof} We first verify that the image of $Y(k_*)$ under the projection
$$\pi:\prod_i{\rm Gr}_{k_{i+1}}(W\otimes t^iA_N)\rightarrow{\rm Gr}_{k}(W\otimes A_N)$$
given by $(U_1,\ldots,U_{N})\mapsto U_1$
is contained in $X_{k,m}^{(N)}$, and even in $C(k_*)$: indeed, the defining relations of $Y(k_*)$ imply that
$$t(U_1)\subset U_2\subset U_1,$$ which shows that $\pi$ projects to $X_{k,m}^{(N)}$. More generally, for all $i=0,\ldots,N-1$,
iterating the defining relations shows that
$$t^{i}(U_1)\subset U_{i+1},$$
which is thus a subspace of dimension at most $k_{i+1}$. But these are precisely the defining conditions of $C(k_*)$.\\
Applying this argument to the special case $k_*=k_*(\lambda)$ for a partition $\lambda$ as before, we see that the image
$\pi(Y(k_*(\lambda)))$ is contained in $C(k_*(\lambda))$, which by the proof of Theorem \ref{adherence} equals the closure of the
orbit $\mathcal{O}_\lambda$.\\
Again by the proof of Theorem \ref{adherence}, we already know that,
for a point $U\in\mathcal{O}_\lambda$, we have $\dim t^i(U)=k_{i+1}(\lambda)$ for $i=0,\ldots,N-1$. Thus, the fibre of $\pi$
over such a point $U$ consists of the single point $(U,tU,\ldots,t^{N-1}U)$ for dimension reasons.
First this proves that $\mathcal{O}_\lambda$ is contained in the image of the map
$\pi:Y(k_*(\lambda))\rightarrow\overline{\mathcal{O}_\lambda}$. But since $\pi$ is proper, thus has closed image,
even $\overline{\mathcal{O}_\lambda}$ is contained in the image, thus (by what is already proven)
$\pi$ maps onto $\overline{\mathcal{O}_\lambda}$. Second, the above argument proves that $\pi$ is generically one-to-one.\\
Finally, we prove that $Y(k_*)$ is always smooth, by realizing it as a tower of Grassmann bundles. To do this,
we consider truncated versions of the variety $Y(k_*)$, namely, for $i=1,\ldots,N$, we consider the subvariety
$Y_i(k_*)$ of $\prod_{j=i-1}^{N-1}{\rm Gr}_{k_{j+1}}(W\otimes t^jA_N)$ given by the same conditions as before, i.e.
the subvariety consisting of tuples $(U_i,\ldots,U_{N})$ such that $tU_j\subset U_{j+1}$ for all relevant $j$ and
$U_i\supset\ldots\supset U_N$. We have a sequence of projections
$$Y(k_*)=Y_1(k_*)\rightarrow Y_2(k_*)\rightarrow\ldots\rightarrow Y_N(k_*).$$
Obviously, we have $$Y_N(k_*)\simeq{\rm Gr}_{k_N}(\underbrace{W\otimes t^{N-1}A_N)}_{\simeq W}).$$
Now we consider the projection $Y_{N-1}(k_*)\rightarrow Y_N(k_*)$, which is equivariant for the action of ${\rm GL}(W\otimes t^{N-1}A_N)$,
the latter being transitive on the target. Its fibre over a point $U_N$ consists of all $U_{N-1}$ such that
$$U_{N}\subset U_{N-1}\subset t^{-1}U_{N}.$$
Note that $U_{N}\subset t^{-1}U_{N}$ since $tU_N=0\subset U_N$. Thus the fibre over $U_N$
is isomorphic to $${\rm Gr}_{k_{N-1}-k_N}(\underbrace{t^{-1}U_{N}/U_N}_{\simeq W}).$$
We continue inductively: the fibre of a projection $Y_{i-1}(k_*)\rightarrow Y_{i}(k_*)$ over a point $U_{i}$ consists of all $U_{i-1}$
such that $U_{i}\subset U_{i-1}\subset t^{-1}U_{1}$; we note that $tU_{i}\subset U_{i+1}\subset U_i$ and thus $U_i\subset t^{-1}U_{i}$.
Thus the fibre over $U_i$ is isomorphic to $${\rm Gr}_{k_{i-1}-k_i}(\underbrace{t^{-1}U_i/U_i}_{\simeq W}).$$
Thus we see that $Y(k_*)$ is a tower of Grassmann bundles with fibres isomorphic to ${\rm Gr}_{k_{i}-k_{i+1}}(W)$
for $i=1,\ldots,N$ (formally defining $k_{N+1}=0$), and in particular, $Y(k_*)$ is irreducible, smooth, of dimension $$mk_1-\sum_i(k_i-k_{i+1})^2.$$
In the special case $k_*=k_*(\lambda)$, this dimension can be written as $mk-\sum_i(\lambda_i')^2$ as expected.
\end{proof}

\subsection{All $X_{k,m}^{(N)}$ are affine Schubert varieties}
Let $P_i$, $i=0,\dots,m$ be the maximal parahoric subgroup of $\widehat{SL}_m$, corresponding to the simple root $\al_i$
of the affine algebra $\widehat{\msl}_m$.
We prove the following proposition.
\begin{prop}\label{Sch}
The variety $X_{k,m}^{(N)}$ is isomorphic to a Schubert variety in $\widehat{SL}_m/P_{k\!\!\!\mod\! m}$.
\end{prop}
\begin{proof}
Recall the Iwahori group $I\subset \widehat{SL}_m$. Consider the $T$-fixed point $x_{k,m}^{(N)}\in X_{k,m}^{(N)}$ such that
$Ix_{k,m}^{(N)}$ is the open part of $X_{k,m}^{(N)}$.

Namely, fixing the standard basis $e_1,\dots,e_m$ of $\bC^m$ let us write
$k=Nr+s$, $0\le j<N$. Then $x_{k,m}^{(N)}$ is the subspace of $\bC^m\T\bC[t]/t^N$ spanned by the vectors $e_i\T t^j$,
$i=1,\dots,r-1$, $j=0,\dots,N-1$ and $e_r\T t^j$, $j=N-s,\dots,N-1$.

Our goal is to prove the existence of a $T$-fixed point $y_{k,m}^{(N)}\in \widehat{SL}_m/P_{k\!\!\!\mod\! m}$ such that
the closure of $Iy_{k,m}^{(N)}$ (the corresponding Schubert variety) is isomorphic to the closure of
$Ix_{k,m}^{(N)}$. To this end we construct an embedding of the $I$-module $\Lambda^k(\bC^m\T \bC[t]/t^N)$ into
the semi-infinite wedge space with the image contained in the level one representation $L_{k\!\!\!\mod\! m}$ corresponding to the
$\msl_m$ fundamental weight $\omega_{k\!\!\!\mod\! m}$.

Let us write $k=ma+b$, where $b=k\!\!\!\mod\! m$. Then $X_{k,m}^{(N)}$ contains the span of
\[
\bC^m \T t^j, j=t^{N-1},\dots,t^{N-a},\ e_1\T t^{N-a-1},\dots, e_b\T t^{N-a-1}.
\]
This point is obviously $I$-fixed. We want to identify it with the highest weight line in $L_b$.
This line inside $F^{(b)}$ is spanned by $\bC^m \T t^j$, $j\ge 0$ and $w_1\T t^{-1},\dots, w_b\T t^{-1}$.
Hence the map $e_i\T t^j\mapsto w_i\T t^{j-N+a}$ induces the $I$-modules embedding of $\Lambda^k(\bC^m\T \bC[t]/t^N)$
to $L_b$, sending $X_{k,m}^{(N)}$ to the Schubert variety $\overline{Iy_{k,m}^{(N)}}\subset \widehat{SL}_m/P_b$.
\end{proof}

\begin{cor}\label{aSv}
${\rm Gr}_N^a(\mgl_n)$ is an affine Schubert variety for $\widehat{SL}_{2n}$.
\end{cor}

\begin{rem}
We note that Proposition \ref{Sch} implies that $X_{k,m}^{(N)}$ is irreducible, normal, 
Cohen-Macaulay with rational singularities 
(see~\cite[Th\'eor\`emes~2.$\Sigma$,~3]{Mat}, \cite[Theorems~2.16,~2.23]{Kum}). 
In particular, this reproves the first part of Corollary \ref{geocon}.
\end{rem}

\section{Affine Lie algebras $\widehat{\msl}_2$ and $\widehat{\msp}_{2n}$}
\subsection{Type $A_1$}
In this subsection we restrict to the case $\g=\msl_2$. We prove Conjecture \ref{T} and give an explicit realization of the
degenerate affine Grassmannian inside the Sato Grassmannian $\rm{SGr}_0$.

Recall the identification $V\simeq W\T \bC[t,t^{-1}]$, $\dim W=2$. Let $\rm{pr}$ be the projection operator along $W\T 1$
to the span of the basis vectors $w_i\T t^j$, $j\ne 1$. We will also need a skew-symmetric form on $V$ defined by
\[
(w_1\T t^i,w_2\T t^j)=\delta_{i+j,-1}, \ (w_1\T t^i,w_1\T t^j)=(w_2\T t^i,w_2\T t^j)=0.
\]
\begin{thm}\label{sl2}
The degenerate affine Grassmannian $\Gr^a(\msl_2)$ sits inside $\rm{SGr}_0$ as the subvariety of subspaces
$U$ satisfying the following conditions
\begin{enumerate}
\item $\rm{pr} (tU)\subset U$,
\item $U$ is isotropic with respect to the above symplectic form.
\end{enumerate}
\end{thm}

We first show the existence of the embedding $\Gr^a(\msl_2)\subset \rm{SGr}_0$ by proving Conjecture \ref{T}.
To do this we prove that there exists a basis of $L_0(\msl_2)$ such that its vectors are linearly independent when considered
in $\tilde L_0(\msl_2)$ (see Section \ref{tilde}). Let $e,h,f$ be the standard basis of $\msl_2$.
For an element $x\in\msl_2$ we set $x_i=x\T t^i$. Recall the construction of the $ehf$-basis of $L_0$ from
\cite{FKLMM}, Theorem 4.
A monomial of the form
\begin{equation}
\label{form}
\dots f_{-n}^{a_n} h_{-n}^{b_n} e_{-n}^{c_n}\dots
f_{-1}^{a_1} h_{-1}^{b_1} e_{-1}^{c_1}
\end{equation}
is called a $ehf$-monomial if it satisfies the following conditions:
\begin{enumerate}
\item[(a)] $a_i+a_{i+1}+b_{i+1}\le 1$ for $i> 0$,
\label{a}
\item[(b)] $a_i+b_{i+1}+c_{i+1}\le 1$ for $i> 0$,
\label{b}
\item[(c)] $a_i+b_i+c_{i+1}\le 1$ for $i> 0$,
\label{c}
\item[(d)] $b_i+c_i+c_{i+1}\le 1$ for $i> 0$.
\label{d}
\end{enumerate}
Then applying the $ehf$-monomials to a highest weight vector of  $L_0$ one gets a basis.
The following picture from \cite{FKLMM} illustrates the set of $ehf$-monomials:
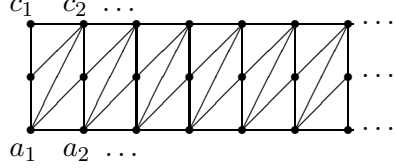
\begin{center}
\begin{picture}(180,70)

\multiput(40,20)(0,40){2}{\line(1,0){122}}
\multiput(40,20)(20,0){7}{\line(0,1){40}}
\multiput(40,20)(20,0){6}{\line(1,1){20}}
\multiput(40,20)(20,0){6}{\line(1,2){20}}
\multiput(40,40)(20,0){6}{\line(1,1){20}}
\multiput(165,20)(0,20){3}{\dots}
\multiput(40,20)(20,0){7}{\circle*{3}}
\multiput(40,40)(20,0){7}{\circle*{3}}
\multiput(40,60)(20,0){7}{\circle*{3}}
\put(32,10){$a_1$}
\put(52,10){$a_2\ \dots$}
\put(32,65){$c_1$}
\put(52,65){$c_2 \ \dots$}
\end{picture}
\end{center}
Namely one considers the set of monomials \eqref{form} such that the sum of exponents over the ends of
any segment is less than or equal to $1$.

\begin{lem}
The monomials \eqref{form} subject to the conditions (a)--(d) form a basis of $\tilde L_0$.
\end{lem}
\begin{proof}
Recall the identification $V\simeq W\T \bC[t,t^{-1}]$ given by $v_{2i+1}\mapsto w_1\T t^{-i-1}$, $v_{2i}\mapsto w_2\T t^{-i}$
(not to be confused with $x_k=x\T t^k$ for $x\in\msl_2$). 
We note that
\begin{gather*}
h_k v_{2i+1} = v_{2i+1+2k},\ h_k v_{2i} = - v_{2i+2k}, \\
e_k v_{2i+1} = 0,\ e_k v_{2i}= v_{2i+2k-1},\\
f_k v_{2i+1} = v_{2i+2k+1},\ f_k v_{2i}= 0.
\end{gather*}
Now assume that we are given a monomial $m$ of the form \eqref{form} subject to the conditions (a)--(d).
Then $m|0\ket$ is decomposed as a sum of several semi-infinite wedge products of vectors $v_i$. We attach
to $m$ one wedge product $w(m)$ from this decomposition. We then show that, given a linear combination of
$ehf$-monomials, we can find a monomial $m$ in it such that $w(m)$ does not show up in any other $ehf$-monomial.

We note that $f_{-1}$ shifts an index by $1$, $h_{-1}$ by $2$, $e_{-1}$ by $3$, $f_{-2}$ by $3$,
$h_{-2}$ by $4$, $e_{-2}$ by $5$ and so on. Now we can see that if $m$ is a $ehf$-monomial, then
all powers $a_i$, $b_i$ and $c_i$ are zeroes or ones  and moreover, if $x_i$ and $y_j$ show up in $m$, then
the difference of their shifts is at least two. Hence $m|0\ket$ contains a semi-infinite wedge
product of the form
\[
E_{i_1,j_1}\dots E_{i_k,j_k} (v_0\wedge v_1\wedge v_2\wedge\dots),\ j_k>j_{k-1}>\dots >j_1>i_1>\dots>i_k,
\]
where $E_{i,j}$ are matrix units. 
Now it is easy to see that given a linear combination of $ehf$-monomials we can
find a wedge product as above, contained in a single $ehf$-monomial.
\end{proof}

Now let us consider the degenerate affine Grassmannian $\Gr^a(\msl_2)$; since affine Grassmannians are special cases
of affine flag varieties, the general Definition \ref{dafv} applies (with $k=1$ and $\la=0$).

\begin{cor}
The $\widehat{\msl}_2$ degenerate Grassmannian is contained in the Sato Grassmannian $\rm{SGr}_0$.
\end{cor}

\begin{lem}\label{lc}
Let $U\in \Gr^a(\msl_2)$. Then $U$ satisfies all the conditions of Theorem \ref{sl2}.
\end{lem}
\begin{proof}
We use the standard basis $w_1,w_2$ of the two-dimensional vector representation of $\msl_2$. 
In particular, $fw_1=w_2$, $ew_2=w_1$, $hw_1=w_1$, $hw_2=-w_2$.
Consider an element $g=\exp(\sum_{i<0} x_ie_{-i}+y_ih_{-i}+z_i f_{-i})$, $g\in \widehat{G}^{-,a}$. Then $g|0\ket$
is spanned by vectors of the form
\begin{gather*}
w_1\T 1 + y_1w_1\T t^{-1} + z_1w_2\T t^{-1} + y_2w_1\T t^{-2}+z_2w_2\T t^{-2} +\dots,\\
w_2\T 1 + x_1w_1\T t^{-1} - y_1w_2\T t^{-1} + x_2w_1\T t^{-2} - y_2w_2\T t^{-2} +\dots,\\
w_1\T t + y_2w_1\T t^{-1} + z_2w_2\T t^{-1} + y_3w_1\T t^{-2}+z_3w_2\T t^{-2} +\dots,\\
w_2\T t + x_2w_1\T t^{-1} - y_2w_2\T t^{-1} + x_3w_1\T t^{-2} - y_3w_2\T t^{-2} +\dots.
\end{gather*}
Now its easy to check that the linear span of these vectors satisfies all the conditions of Theorem \ref{sl2}. Since
$\Gr(\msl_2)$ is the closure of the $\widehat{G}^{-,a}$ orbit, the lemma follows.
\end{proof}

\begin{cor}
Theorem \ref{sl2} holds.
\end{cor}
\begin{proof}
One sees from the proof of Lemma \ref{lc} that the $\widehat{G}^{-,a}$ orbit of the line $\bC|0\ket$ consists of the subspaces
$U$ satisfying the conditions of Theorem \ref{sl2} and such that $U$ has a nontrivial Pl\"ucker coordinate corresponding to the subspace
$W\T\bC[t]$. Since this is an open condition, Theorem \ref{sl2} follows.
\end{proof}

Finally, let us compute the torus acting on the degenerate affine Grassmannian. Assume that we have a torus scaling the basis
vectors as $w_1\T t^i\to q_iw_1\T t^i$, $w_2\T t^j\to p_jw_2\T t^j$ for some numbers $p_i$, $q_j$. We want this torus to act on the
open cell.
In particular, each vector in the list of vectors from the proof of Lemma \ref{lc} has to be invariant (up to scaling)
with respect to the torus action. This gives the following set of relations, labeled by positive numbers $k$:
\begin{gather*}
\frac{p_0}{q_{-k}}=\frac{p_1}{q_{-k+1}}=\dots = \frac{p_{k-1}}{q_{-1}},\\
\frac{q_0}{q_{-k}}=\frac{q_1}{q_{-k+1}}=\dots = \frac{q_{k-1}}{q_{-1}}=\frac{p_0}{p_{-k}}=\frac{p_1}{p_{-k+1}}=\dots = \frac{p_{k-1}}{p_{-1}},\\
\frac{q_0}{p_{-k}}=\frac{q_1}{p_{-k+1}}=\dots = \frac{q_{k-1}}{p_{-1}}.
\end{gather*}
\begin{rem}
The values $p_k=p_1r^{k-1}$, $q_k=q_1r^{k-1}$ for arbitrary $p_1,q_1,r$ solve the equations above.
This is a torus (effectively two-dimensional, since it contains one-dimensional torus $r=1$, $p_1=q_1$, scaling all the vectors by the same number), 
generated by the loop rotation and the Cartan torus of $SL_2$.
\end{rem}
The system above is equivalent to the statement that for any $a\ge 0$, $b<0$ the quantities
$\frac{p_a}{q_b}$, $\frac{q_a}{p_b}$, $\frac{p_a}{p_b}$, $\frac{q_a}{q_b}$ depend only on the difference $a-b$ and
$\frac{p_a}{p_b}=\frac{q_a}{q_b}$. This means that there exists a complex number $r$ such that
\[
q_a=q_0r^a, p_a=p_0r^a, a\ge 0;\quad q_b=q_{-1}r^{b+1}, p_b=p_{-1}r^{b+1}, b<0
\]
with additional condition $\frac{q_0}{q_{-1}}=\frac{p_0}{p_{-1}}$.
\begin{cor}
The torus acting on the degenerate Grassmannian is 3-dimensional. It is generated by the loop rotation, the one-dimensional Cartan torus
of $SL_2$ and an additional one-dimensional torus with $p_{\ge 0}=q_{\ge 0}=1$, $p_{<0}=q_{<0}=\mathrm{const}.$
\end{cor}

\subsection{Symplectic degenerate affine Grassmannian}
It turns out that the construction of ${\rm Gr}^a(\msl_2)$ has a natural (though conjectural) generalization to the case
of symplectic algebras.

Let $W$ be a $2n$-dimensional vector space endowed with a non-degenerate skew-symmetric form $(\cdot,\cdot)$. We fix a basis
$w_1,\dots,w_{2m}$ of $W$ such that $(w_i,w_{2n+1-i})=1$ for $i=1,\dots,n$. Consider the space
$W\T \bC[t,t^{-1}]$ and the corresponding sector $F^{(0)}$ of the semi-infinite wedge power. 

Now consider the Lie algebra $\msp^a_{2n}$, which is an abelian Lie algebra with the underlying vector space $\msp_{2n}$.
Define the action of $\msp_{2n}^a\T t^{-1}\bC[t^{-1}]$ on $W\T \bC[t,t^{-1}]$ as follows. For $x\in \msp_{2n}^a$, $w\in W$ we define
\[
(x\T t^i)(w\T t^j)=
\begin{cases}
xw\T t^{i+j}, \text{ if } j\ge 0, i+j<0,\\
0,\text{ otherwise}.
\end{cases}
\]

Now let $L_0(\msp_{2n})$ be the basic level one module of the affine Lie algebra $\widehat{\msp}_{2n}$ and let $L_0^a$ be its degenerate analogue.
Let $l_0\in L_0$ be the highest weight vector.
In particular, $L_0(\msp_{2n})=\U(\msp_{2n}\T t^{-1}\bC[t^{-1}]) l_0$.
\begin{conj}\label{sc}
Let $\tilde L_0(\msp_{2n})=\U(\msp_{2n}^a\T t^{-1}\bC[t^{-1}])|0\rangle$. Then
we have an isomorphism of $\msp_{2n}^a\T t^{-1}\bC[t^{-1}]$-modules
\[
\tilde L_0(\msp_{2n})\simeq L^a_0,\ |0\rangle\mapsto l_0.
\]
\end{conj}

The action of the abelian Lie algebra $\msp_{2n}^a\T t^{-1}\bC[t^{-1}]$ as above induces an action of the corresponding infinite-dimensional Lie group
$\exp(\msp_{2n}^a\T t^{-1}\bC[t^{-1}])$.
Because of Conjecture \ref{sc} the natural candidate for the degenerate symplectic affine Grassmannian
${\rm Gr}^a(\widehat{\msp}_{2n})$ is the closure of the orbit  of this group through the highest weight vector inside $\bP(\tilde L_0(\msp_{2n}))$.
We denote this closure by ${\rm G}^a(\widehat{\msp}_{2n})$.

Define a symplectic form on the infinite-dimensional space $W\T \bC[t,t^{-1}]$: $\bra v\T t^i,w\T t^j\ket =(v,w)\delta_{i+j,-1}$.
\begin{thm}\label{sp}
The variety ${\rm G}^a(\widehat{\msp}_{2n})$ consists of points $U$ of the Sato Grassmanian $SGr_0$ such that $U$ is isotropic with respect
to the form $\bra\cdot ,\cdot\ket$ and $pr(tU)\subset U$.
\end{thm}

\begin{conj}
${\rm Gr}^a(\widehat{\msp}_{2n})\simeq {\rm G}^a(\widehat{\msp}_{2n})$.
\end{conj}

We sketch the proof of Theorem \ref{sp}.
Let $\rm{O}\subset SGr_0$ be the subvariety of spaces that intersect trivially with $W\T t^{-1}\bC[t^{-1}]$
(i.e. the Pl\"ucker coordinate corresponding to $W\T\bC[t]$ does not vanish).
\begin{prop}
The $\exp(\msp_{2n}\T t^{-1}\bC[t^{-1}])\cdot \bC |0\ket$ orbit coincides with $\rm{O}\cap \overline{{\rm G}^a(\widehat{\msp}_{2n})}$.
\end{prop}

In order to prove Theorem \ref{sp} we consider the finitization of ${\rm G}^a(\widehat{\msp}_{2n})$, thus making explicit the
ind-variety structure.
\begin{dfn}
For $N\ge 0$, let ${\rm G}_N^a(\widehat{\msp}_{2n})$  be the finite-dimensional subvariety of ${\rm G}^a(\widehat{\msp}_{2n})$
consisting of subspaces $U$ such that
\begin{enumerate}
\item $W\T t^N\bC[t]\subset U\subset W\T t^{-N}\bC[t]$,
\item $U$ is isotropic,
\item $pr(tU)\subset U$.
\end{enumerate}
\end{dfn}

\begin{lem}
${\rm G}_N^a(\widehat{\msp}_{2n})$ coincides with the closure of the orbit of the group $\exp(\msp_{2n}^a\T {\rm span}(t^{-1},\dots,t^{-N}))$
through the line spanned by $|0\ket$.
In particular, ${\rm G}_N^a(\widehat{\msp}_{2n})$ are irreducible and of dimension $N\dim\msp_{2n}$.
\end{lem}

\begin{lem}
${\rm G}_N^a(\widehat{\msp}_{2n})\cap O = \exp(\msp_{2n}^a\T {\rm span}(t^{-1},\dots,t^{-N}))\cdot\bC |0\ket\cap O$.
\end{lem}

Now it remains to prove the irreducibility of the varieties ${\rm G}_N^a(\widehat{\msp}_{2n})$.
This is achieved by constructing explicitly the desingularization via the same procedure as in
Lemma \ref{irr} and section \ref{resolution} (see also \cite{FFiL}, Definition 5.1).

\section*{Acknowledgments}
Thanks are due to L.~Positselski for his explanations about the notion of
flatness. We are very grateful to B.~Feigin for extremely fruitful discussions
of degenerate affine Grassmannians.
The work of E.F. was partially supported
by the Dynasty Foundation and by the Simons foundation. 
The work of EF was supported within the framework of a subsidy granted to the HSE 
by the Government of the Russian Federation for the implementation of the Global Competitiveness Program.
The research of M.F. was carried out at the IITP RAS at the expense of the 
Russian Foundation for Sciences (project no. 14-50-00150).

\end{document}